\documentclass[11pt, leqno]{article}
\date{}
\usepackage{amssymb, amsbsy, amsmath, amsfonts, amssymb, amscd, mathrsfs, amsthm}
\usepackage[english]{babel}
\usepackage[T1]{fontenc}
\usepackage{indentfirst}
\usepackage{color}

\makeatletter
\@addtoreset{equation}{section}
\makeatother

\newtheorem{statement}{}[section]
\newtheorem{theoreme}[statement]{Theorem}
\newtheorem{lemme}[statement]{Lemma}

\newtheorem{proposition}[statement]{Proposition}

\newcommand\C{\mathbb C}
\newcommand\N{\mathbb N}
\newcommand\R{\mathbb R}
\newcommand\T{\mathbb T}
\newcommand\D{\mathbb D}

\newcommand\e{{\rm e}}

\newcommand\eps{\varepsilon}
\newcommand\ind{{\rm 1\kern-.30em I}}

\renewcommand \Re{{\mathfrak R}{\rm e}\,}
\renewcommand \Im{{\mathfrak I}{\rm m}\,}
\let\phi=\varphi

\title{\bf Composition operators with surjective symbol and small approximation numbers}
\author{\it Daniel Li,  Herv\'e Queff\'elec, Luis Rodr{\'\i}guez-Piazza }

\date{\footnotesize \today}

\begin{document}

\maketitle

\noindent {\bf Abstract.} We give a new proof of the existence of a surjective symbol whose associated composition operator on $H^2 (\D)$ is in all 
Schatten classes, with the improvement that its approximation numbers can be, in some sense, arbitrarily small. We show, as an application, that, contrary to the 
$1$-dimensional case, for $N \geq 2$, the behavior of the approximation numbers $a_n = a_n (C_\phi)$, or rather of  
{$\beta^-_N = \liminf_{n \to \infty} [a_n]^{1/ n^{1/ N}}$ or $\beta^+_N = \limsup_{n \to \infty} [a_n]^{1/ n^{1/ N}}$}, of composition operators on 
$H^2 (\D^N)$ cannot be determined by the image of the symbol.
\medskip

\medskip


%%%%%%%%%%%%%%%%%%%%%%%%%%%%%%%%%%%%%%%%%%
\section{Introduction}

We start by recalling some notations and facts. 

Let $\D$ be the open unit disk, $H^2$ the Hardy space on $\D$, and  $\varphi \colon \D\to \D$  a non-constant analytic self-map. It is well-known 
(\cite{SHA}) that $\varphi$ induces a composition operator $C_\varphi \colon H^2 \to H^2$ by the formula:
\begin{displaymath} 
C_{\varphi} (f) = f \circ \varphi \, ,
\end{displaymath} 
and the connection between the ``symbol'' $\varphi$ and the properties of the operator $C_{\varphi} \colon H^2 \to H^2$, in particular its compactness, can be 
further studied (\cite{SHA}).

We also recall that the $n$th approximation number $a_{n}(T)$, $n = 1, 2, \ldots$, of an  operator $T \colon H_1 \to H_2$, between  Hilbert spaces $H_1$ and 
$H_2$, is defined as the distance of $T$ to operators of rank $ < n$, for the operator-norm:
\begin{equation} 
a_n (T) = \inf_{{\rm rank}\, R \, < n} \| T - R\| \, .
\end{equation} 

The $p$-Schatten class $S_{p} (H_1, H_2)$, $p > 0$ consists of all $T \colon H_1 \to H_2$ such that $\big( a_{n} (T) \big)_n \in \ell^p$. The approximation 
numbers have the ideal property:
\begin{displaymath} 
a_{n} (A T B) \leq \Vert A\Vert \, a_{n} (T) \, \Vert B \Vert \, . 
\end{displaymath} 

Let now, for $\xi \in \T = \partial{\D}$ and $h > 0$, the Carleson window $S (\xi,h)$ be defined as:
\begin{equation} 
S (\xi,h) = \{z \in \D \, ; \ |z - \xi| \leq h \} \, . 
\end{equation} 
For a symbol $\varphi$, we define $m_\varphi = \varphi^{\ast} (m)$ where $m$ is the Haar measure of $\T$ and 
$\varphi^{\ast} \colon \T\to \overline{\D}$ the (almost everywhere defined) radial limit function associated with $\varphi$, namely: 
\begin{displaymath} 
\varphi^{\ast} (\xi) = \lim_{r \to 1^{-}} \varphi (r\xi) \, .
\end{displaymath} 

Finally, we set for $h > 0$:
\begin{equation} 
\rho_{\varphi} (h) = \sup_{\xi\in \T} m_{\varphi} [S (\xi, h)] \, .
\end{equation} 

It is known (\cite{SHA}) that $\rho_{\varphi} (h) = {\rm O}\, (h)$ and (\cite{MCCL}) that $C_\varphi$ is compact if and only if  
$\rho_{\varphi} (h) = {\rm o}\, (h)$ as $h \to 0$. Simpler criteria (\cite{SHA}) exist when $\varphi$ is injective, or even $p$-valent, meaning that for any 
$w \in \D$, the equation $\varphi (z) = w$ has at most $p$ solutions. 

A measure $\mu$ on $\D$ is called $\alpha$-Carleson, $\alpha \geq 1$, if $\sup_{|\xi| = 1} \mu [ S (\xi, h)] = {\rm O}\, (h^\alpha)$.
\smallskip

B.~MacCluer and J.~Shapiro showed in \cite[Example~3.12]{MCSH} the following result, paradoxical at first glance.
\begin{theoreme} [MacCluer-Shapiro] \label{jb} 
There exists a \emph{surjective and four-valent} symbol $\varphi \colon \D\to \D$ such that the composition operator $C_\varphi \colon H^2 \to H^2$ is  
compact.
\end{theoreme}

Observe that such a symbol $\varphi$ cannot be one-valent (injective), because it would be an automorphism of $\D$, and $C_\varphi$ would be invertible and 
therefore not compact. In \cite[Theorem~4.1]{LLQR}, we gave the following improved statement.
\begin{theoreme} \label{first} 
For every non-decreasing function $\delta \colon (0, 1) \to (0, 1)$, there exists a two-valent symbol and nearly surjective (i.e. 
$\varphi (\D) = \D \setminus \{0\}$) symbol $\phi$, and  $0 < h_0 < 1$,  such that:
\begin{equation} \label{firs} 
\quad m(\{ z \in \T \, ; \ |\phi^\ast (z)| \geq 1 - h\}) \leq \delta (h)  \quad \text{for } 0 < h \leq h_0 \, . 
\end{equation}
As a consequence, there exists a \emph{surjective and four-valent} symbol $\psi \colon \D \to \D$ such that the composition operator 
$C_\psi \colon H^2 \to H^2$ is in every Schatten class $S_{p}(H^2)$, $p > 0$.
\end{theoreme}

Our proof was rather technical and complicated, and based on arguments of barriers and harmonic measures. 
\smallskip

The goal of this paper is to give a more precise statement of Theorem~\ref{first} in terms of approximation numbers $a_{n} (C_\varphi)$, and not only in terms 
of Schatten classes, and with a simpler proof. We then apply this result to show that for the polydisk $\D^N$, $N \geq 2$, the nature (boundedness, compactness, 
asymptotic behavior of approximation numbers) of the composition operator cannot be determined by the geometry of the image $\phi (\D^N)$ of its symbol 
$\phi$. For certain asymptotic behavior of approximation numbers, this is contrary to the $1$-dimensional case (see \cite[Theorem~3.1 and Theorem~3.14]{SRF}).
\smallskip

The notation $A \lesssim B$ means that $A \leq C \, B$ for some positive constant $C$, and $A \approx B$ that $A \lesssim B$ and $B \lesssim A$.

%%%%%%%%%%%%%%%%%%%%%%%%%%%%%%%%%%%%%%%%%%%%%%%%%%%%%%%%%%%%%%%%%%%%%%%%%
\section{Background and preliminary results} 

We initiated the study of approximation numbers of composition operators on  $H^2$ in \cite{LIQUEROD},  and proved the following basic results:  
\begin{theoreme}\label{basic}  
If $\varphi$ is any symbol, then, for some $\delta > 0$ and $r > 0$, or $a > 0$:
\begin{displaymath} 
a_{n} (C_\varphi) \geq \delta \, r^n = \delta\, \e^{- a n} \, .
\end{displaymath} 
Moreover, as soon as $\Vert \varphi \Vert_\infty = 1$, there exists some sequence $\varepsilon_n$ tending to $0$ such that:
\begin{displaymath} 
a_{n} (C_\varphi) \geq \delta \, \e^{- n \varepsilon_n} \, .
\end{displaymath} 
\end{theoreme}

We also   proved  in \cite[Theorem 5.1]{LIQUEROD} that:
\begin{proposition} \label{second} 
For any symbol $\varphi$, we have:
\begin{displaymath} 
a_{n} (C_\varphi) \lesssim \inf_{0 < h < 1} \bigg[ \e^{- n h} + \sqrt{\frac{\rho_{\varphi}(h)}{h}} \bigg] \, .
\end{displaymath} 
\end{proposition}

We also recall (see \cite{LIQUEROD}) that, for $\gamma > - 1$, the weighted Bergman space $\mathcal{B}_{\gamma}$ is the space of functions 
$f (z) = \sum_{n = 0}^\infty a_n z^n$ such that:
\begin{equation} 
\Vert f \Vert_{\gamma}^{2} := \sum_{n = 0}^\infty \frac{|a_n|^2}{(n+1)^{\gamma+1}} < \infty \, .
\end{equation} 
Equivalently, $\mathcal{B}_{\gamma}$ is the space of analytic functions $f \colon \D \to \C$ such that:
\begin{equation} 
\int_\D |f (z)|^2 \, (\gamma + 1) (1 - |z|^2)^\gamma \, dA (z) < \infty \, ,
\end{equation} 
where $dA$ is the normalized area measure on $\D$, and then:
\begin{equation} 
\int_\D |f (z)|^2 \, (\gamma + 1) (1 - |z|^2)^\gamma \, dA (z) \approx \Vert f \Vert_{\gamma}^{2} \, .
\end{equation} 

The case $\gamma = 0$ corresponds to the usual Bergman space ${\mathcal B}^2$, and the limiting case $\gamma = - 1$ to the Hardy space $H^2$. We wish to 
note in passing (we will make use of that elsewhere) that the proof of Theorem 5.1 in \cite{LIQUEROD} easily gives the following result.
\begin{proposition} \label{seconde} 
Let $ \gamma > - 1$ and $\varphi$ a symbol inducing a bounded composition operator $C_\varphi \colon \mathcal{B}_{\gamma} \to H^2$. Then:
\begin{displaymath}
a_{n} (C_\varphi \colon \mathcal{B}_{\gamma} \to H^2) \lesssim \inf_{0 < h < 1} 
\Bigg( {(n + 1)^{(\gamma + 1)/2}} \, \e^{- n h}  + \sup_{0 < t \leq h} \sqrt{\frac{\rho_{\phi} (t)}{t^{2 + \gamma}}} \,\, \Bigg) \, \cdot
\end{displaymath} 
\end{proposition}
\begin{proof} 
Take $E = z^n \mathcal{B}_{\gamma}$; this is a subspace of $\mathcal{B}_{\gamma}$ of codimension $\leq n$. Let $f \in E$ with 
$\Vert f\Vert_{\gamma} = 1$. Writing $f = z^n g$ with $\Vert g \Vert_{\gamma}^2 \leq (n + 1)^{\gamma + 1}$ and splitting the integral into two parts, 
we have, for $0 < h < 1$:
\begin{displaymath} 
\| C_{\varphi} f \|_{H^2}^{2} = \int_{\D} |f|^2 \, dm_\phi  
\leq (1 - h)^{2 n} \int_{(1 - h) \D} |g|^2 \, dm_\phi + \int_{\D \setminus (1 - h) \D} |f|^2 \, dm_\phi \, .
\end{displaymath} 
For the first integral, we have:
\begin{equation} \label{Luis 1}
\int_{(1 - h) \D} |g|^2 \, dm_\phi \leq \int_{\D} | g |^2 \, dm_\phi = \| C_\phi \, g \|_{H^2}^2 
\leq \| C_\phi\|_{\mathcal{B}_{\gamma} \to H^2}^2 \| g \|_\gamma^2  \, .
\end{equation} 
For the second integral, we have:
\begin{displaymath} 
\int_{\D \setminus (1 - h) \D} |f|^2 \, dm_\phi \leq \| J \colon \mathcal{B}_{\gamma} \to L^{2} (\mu_h) \|^2 \, ,
\end{displaymath} 
where $\mu_h$ is the restriction of $m_\varphi$ to the annulus $\{z \in \D \, ; \ 1 - h < |z| < 1 \}$ and $J$ the canonical injection of 
$\mathcal{B}_{\gamma}$ into $L^{2} (\mu_h)$. Hence Stegenga's version of the Carleson embedding theorem for $\mathcal{B}_{\gamma}$ 
(\cite[Theorem~1.2]{Stegenga}; see \cite {Hastings} for the unweighted case; see also \cite[p.~62]{DS} or \cite[ p.~167]{ZH}) gives us:
\begin{equation} \label{Luis 2}
\int_{\D \setminus (1 - h) \D} |f|^2 \, dm_\phi \lesssim \sup_{0 < t \leq h} \frac{\rho_{\phi} (t)}{t^{2 + \gamma}} \, \cdot
\end{equation} 
Putting \eqref{Luis 1} and \eqref{Luis 2} together, that gives:
\begin{displaymath} 
\| C_{\varphi} f \|_{H^2} \lesssim \e^{- n h} (n + 1)^{(\gamma + 1)/2}  + \sup_{0 < t \leq h} \sqrt{\frac{\rho_{\phi} (t)}{t^{2 + \gamma}}} \, \cdot
\end{displaymath} 

In other terms, using the Gelfand numbers $c_k$: 
\begin{displaymath} 
c_{n + 1} (C_\phi \colon \mathcal{B}_{\gamma} \to H^2) \lesssim  (n + 1)^{(\gamma + 1)/2} \, \e^{- n h}
+ \sup_{0 < t \leq h} \sqrt{\frac{\rho_{\phi} (t)}{t^{2 + \gamma}}} \, \cdot
\end{displaymath} 
As $a_{n + 1} = c_{n + 1}$ and as we can ignore the difference between $a_n$ and $a_{n + 1}$, that finishes the proof.
\end{proof}

As an application, we mention the following result. We refer to \cite[Section~4.1]{LIQURO} for the definition of the cusp map, denoted $\chi$.
\begin{theoreme} \label{cusp}  
Let $\chi \colon \D \to \D$ be the cusp map and $\Phi \colon \D^N\to \D^N$ the diagonal map defined by:
\begin{equation} 
\Phi (z_1, z_2, \ldots, z_N) = \big( \chi (z_1), \chi (z_1), \ldots, \chi (z_1) \big) \, . 
\end{equation} 
Then, the composition operator $C_\Phi$ maps $H^{2}(\D^N)$ to itself and:
\begin{equation} \label{new} 
a_{n} (C_\Phi) \lesssim \e^{- d \sqrt{n}} 
\end{equation}
where $d$ is a positive constant depending only on $N$. 
\end{theoreme}
\smallskip

\noindent {\bf Remark.} We have to compare with \cite[Theorem~6.2]{BLQR} where, for:
\begin{displaymath} 
\Psi (z_1, \ldots, z_N) = \big( \chi (z_1), \ldots, \chi (z_N) \big) \, ,
\end{displaymath} 
it is shown that, for constants $b \geq a > 0$ depending only on $N$:
\begin{displaymath} 
\e^{- b \, (n^{1 /N} / \log n)} \lesssim a_n (C_\Psi) \lesssim \e^{- a \, (n^{1 /N} / \log n)} \, .
\end{displaymath} 

Note also that for $N = 1$, the estimate of Theorem~\ref{cusp} is very crude. 

\begin{proof} [Proof of Theorem~\ref{cusp}] 
Take $\gamma = N - 2$. As in \cite[Section~4]{DHL}, we have thanks to the Cauchy-Schwarz inequality, and the fact that 
$\sum_{|\alpha|= n} 1 \approx ( n + 1)^{N - 1}$, a  factorization:
\begin{displaymath} 
C_\Phi = J C_{\chi} M \, ,
\end{displaymath} 
where $M \colon H^{2} (\D^N) \to \mathcal{B}_{\gamma}$ is defined by $M \! f = g$ with:
\begin{equation} \label{M}
\qquad g (z) = f (z, z, \ldots, z) = \sum_{n = 0}^\infty \Bigg( \sum_{|\alpha| = n} a_{\alpha} \Bigg) \, z^n \, , \qquad z \in \D \, , \quad
\end{equation} 
for
\begin{displaymath} 
f (z_1, z_2, \ldots, z_N) = \sum_{\alpha} a_{\alpha} z_1^{\alpha_1} \cdots z_N^{\alpha_N} \, ,
\end{displaymath} 
and where $J \colon H^{2}(\D) \to H^{2} (\D^N)$ is the canonical injection given by:
\begin{equation} \label{J}
(J h) (z_1, z_2, \ldots, z_N) = h (z_1) \, .
\end{equation} 
This corresponds to a diagram:
\begin{equation} \label{factorization}
H^{2} (\D^N) \mathop{\longrightarrow}^M \mathcal{B}_{\gamma}
\mathop{\longrightarrow}^{C_\chi} H^{2} (\D) 
\mathop{\longrightarrow}^J H^{2} (\D^N) \, ,
\end{equation} 
where $C_\chi \colon \mathcal{B}_{\gamma} = \mathcal{B}_{N - 2} \to  H^{2} (\D)$ is a bounded operator. Indeed, 
we have the behavior (\cite[Lemma~4.2]{LIQURO}):
\begin{displaymath} 
|1 - \chi^\ast (\e^{i \theta})| \approx \frac{1}{\log (1/|\theta|)} \, \raise 1,5 pt \hbox{,}
\end{displaymath} 
and this implies, with $c$ an absolute constant:
\begin{equation} \label {chi Carleson}
\begin{split}
m_\chi [S (\xi, h)] & \lesssim m_{\chi} [S (1, h)] = m (\{ |\chi^\ast (\e^{i \theta}) - 1 | < h)  \\
& \lesssim m [ \{ c / \log (1 / |\theta|)  < h \} ]  \leq \e^{- c/h} \, ; 
\end{split}
\end{equation}
in particular $\rho_\chi (h) \leq \e^{- c/h} = {\rm O}\, (h^N)$, so $m_\chi$ is an $N$-Carleson measure and the Stengenga-Carleson theorem 
(\cite[Theorem~1.2]{Stegenga}) says that the operator $C_\chi \colon \mathcal{B}_{N - 2} \to  H^{2} (\D)$ is bounded. 

Now Proposition~\ref{seconde} with \eqref{chi Carleson} give: 
\begin{displaymath} 
a_{n} (C_{\chi} \colon \mathcal{B}_{\gamma} \to H^2) 
\lesssim  \inf_{0 < h < 1} \big[ (n + 1)^{(N - 1)/2} \, \e^{- n h} +  \e^{- c/h} h^{- N/2} \big] \, .
\end{displaymath} 
Adjusting $h = 1/\sqrt{n}$, we get $a_{n} (C_{\chi} \colon \mathcal{B}_{\gamma} \to H^2) \lesssim \e^{- d\sqrt{n}}$ for some positive constant $d$. 
Finally, the factorization $C_\Phi = J C_{\chi} M$ and the ideal property of approximation numbers give the result.
\end{proof}

In the case of lens maps, Proposition~\ref{seconde} gives very poor estimates. We avoid using this theorem in \cite[Section~4]{DHL}, when $N = 2$, 
using the semi-group property of those lens maps. The same proof gives for arbitrary $N \geq 2$ the following result.
\goodbreak
\begin{theoreme} \label{lens} 
Let $\lambda_{\theta}$ the lens map with parameter $\theta$, $0 < \theta < 1$, and let $\Phi \colon \D^N \to \D^N$ be the diagonal map defined by:
\begin{equation} 
\Phi (z_1, z_2, \ldots, z_N) = \big( \lambda_{\theta} (z_1), \lambda_{\theta} (z_1), \ldots, \lambda_{\theta} (z_1) \big) \, . 
\end{equation} 
Then: 
\begin{itemize}
\setlength\itemsep{2 pt}
\item [$1)$] if $\theta >1/N$, $C_\Phi$ is unbounded on $H^{2} (\D^N)$;

\item [$2)$] if $\theta = 1/N$, $C_\Phi$ is bounded and not compact on $H^{2} (\D^N)$;

\item [$3)$] if $\theta < 1/N$, $C_\Phi$ is compact on $H^{2} (\D^N)$ and moreover:
\begin{equation} \label{mino lens}
a_{n} (C_\Phi) \lesssim \e^{- d \sqrt{n}} 
\end{equation} 
for a constant $d > 0$ depending only on $\theta$ and $N$.
\end{itemize}
\end{theoreme}

\noindent {\bf Remark.} In \cite[Theorem~6.1]{BLQR}, it is shown that, for:
\begin{displaymath} 
\Psi (z_1, \ldots, z_N) = \big( \lambda_\theta (z_1), \ldots, \lambda_\theta (z_N) \big) \, ,
\end{displaymath} 
we have, for constants $b \geq a > 0$, depending only on $\theta$ and $N$:
\begin{displaymath} 
\e^{ - b \, n^{1 / (2 N)}} \lesssim a_n (C_\Psi) \lesssim \e^{ - a \, n^{1 / (2 N)}} \, .
\end{displaymath} 
\begin{proof} [Proof of Theorem~\ref{lens}] 
That had been proved, for $N =2$ in \cite[Theorem~4.2 and Theorem~4.4]{DHL}. For convenience of the reader, we sketch the proof.
\smallskip

Assume first $\theta \leq 1/N$, and write $\lambda_\theta = \lambda_{N \theta}\circ \lambda_{1/N}$, where we set, for convenience, $\lambda_1 (z) = z$, 
so $C_{\lambda_1} = {\rm Id}$. 
As in the proof of Theorem~\ref{cusp} (see \cite[Section~4]{DHL}), we have a  factorization:
\begin{displaymath} 
C_\Phi = J C_{\lambda_{N \theta}} C_{\lambda_{1/N}} M \, ,
\end{displaymath} 
where $M$ and $J$ are defined in \eqref{M} and \eqref{J}.
  
This corresponds to a diagram (recall that $\gamma = N - 2$):
\begin{displaymath} 
H^{2} (\D^N) \mathop{\longrightarrow}^M \mathcal{B}_{\gamma}
\mathop{\longrightarrow}^{C_{\lambda_{1/N}}} H^{2} (\D) 
\mathop{\longrightarrow}^{C_{\lambda_{N\theta}}} H^{2} (\D)
\mathop{\longrightarrow}^{J} H^{2} (\D^N) \, .
\end{displaymath} 
The second arrow is bounded, since we know (\cite[Lemma~3.3]{LELIQURO}) that the pullback measure $m_{\lambda_{1/N}}$ is $N$-Carleson, so that 
$C_{\lambda_{1/N}}$ maps $\mathcal{B}_{N - 2}$ to $H^{2}(\D)$ by the Stegenga-Carleson embedding theorem (\cite[Theorem~1.2]{Stegenga}).
\smallskip

For $\theta < 1/ N$, we have $N \theta < 1$ and $C_{\lambda_{N \theta}}$ is compact and, for some constant $b = b (\theta)$, we have 
$a_n (C_{\lambda_{N \theta}}) \lesssim \e^{ - b \sqrt{n}}$ (\cite[Theorem~2.1]{LELIQURO}). Hence $C_\Phi$ is compact and 
$a_n (C_\Phi) \lesssim \e^{ - b \sqrt{n}}$.
\smallskip

Now, for $\theta \geq 1 / N$, we consider the reproducing kernels:
\begin{displaymath} 
K_{a_1, \ldots, a_N} (z_1, \ldots, z_N) = \prod_{j = 1}^N \frac{1}{1 - \overline{a}_ j z_j} \, \cdot
\end{displaymath} 
We have:
\begin{displaymath} 
\| K_{a_1, \ldots, a_N} \|^2 = \prod_{j = 1}^N \frac{1}{1 - |a_j|^2} 
\end{displaymath} 
and:
\begin{displaymath} 
C_\Phi^\ast (K_{a_1, \ldots, a_N}) = K_{\lambda_\theta (a_1), \ldots, \lambda_\theta (a_1)} \, ,
\end{displaymath} 
so:
\begin{displaymath} 
\| C_\Phi^\ast (K_{a_1, \ldots, a_N}) \|^2 = \bigg( \frac{1}{1 - |\lambda_\theta (a_1) |^2} \bigg)^N \, \cdot
\end{displaymath} 
Since:
\begin{displaymath} 
1 - |\lambda_\theta (a_1) |^2 \approx 1 - |\lambda_\theta (a_1) | \approx (1 - |a_1|)^\theta \, , 
\end{displaymath} 
we see that $\| C_\Phi^\ast (K_{a_1, \ldots, a_N}) \| / \| K_{a_1, \ldots, a_N}\|$ is not bounded  for $\theta > 1 / N$, so $C_\phi$ is then not bounded; 
and it does not converge to $0$ for $\theta = 1 / N$, so $C_\Phi$ is then not compact.  
\end{proof}
%
%%%%%%%%%%%%%%%%%%%%%%%%%%%%%%%%%%%%%%%%%%%%%%%%%%%%%%%%%%%%%%%%%%%%%%%%%%%
\section{Surjectivity}

Let us come back to our surjectivity issues. 
\smallskip

Let us first remark that Theorem~\ref{first} gives the following result.
\begin{theoreme} \label{bet} 
For every non-decreasing function $\delta \colon (0, 1) \to (0, 1)$, there exists a \emph{surjective and four-valent} symbol $\psi$, and  $0 < h_0 < 1$, such that, 
for $0 < h \leq h_0$:
\begin{equation} \label{be} 
m (\{z \in \T \, ; \ |\phi^\ast (z)| \geq 1 - h\}) \leq \delta (h) \, . 
\end{equation}
\end{theoreme} 
\begin{proof} 
Just observe that the passage from ``$\varphi$ two-valent and nearly surjective'' to ``$\psi$ four-valent and  surjective'' is harmless: for this, consider the Blaschke 
product:
\begin{displaymath} 
B (z) = \bigg( \frac{z - a}{1 - a z} \bigg)^2 \, ,
\end{displaymath} 
where $0 < a < 1$, and take  $\psi = B\circ \varphi$; we observe that $B (\D \setminus \{0\}) = \D$ since 
$a^2 = B \big( \frac{2 a}{1 + a^2} \big)$, and, for $z \in \D$:
\begin{displaymath} 
\frac{1 - |B (z)|}{1 - |z|}\geq \frac{\displaystyle 1 - |\frac{z - a}{1 - a z}|^2}{1 - |z|^2} = \frac{1 - a^2}{|1 - a z|^2} \geq  \frac{1 - a^2}{4} 
\, \raise 1pt \hbox{,} 
\end{displaymath} 
so that:
\begin{displaymath} 
m (|\psi^\ast| > 1 - h) = m (1 - |B \circ \varphi^\ast| < h) \leq m \big( 1 - |\varphi^\ast| \leq \kappa_a h \big) \, ,
\end{displaymath} 
with $\kappa_a = 4 / (1 - a^2)$. Hence, this map $\psi$ is surjective, four-valent,  and satisfies \eqref{be}, as well, up to a change of $\delta (h)$ to 
$\delta (h /\kappa_a)$ for $\varphi$ at the beginning.
\end{proof}
%
 
%%%%%%%%%%%%%%%%%%%%%%%%%%%%%%%%%%%%%%%%%%%%%%%%%%%%%%%%%%%%%%%%%%%%%%%%%%
\subsection{A more precise statement}

Our new statement is as follows.
\begin{theoreme} \label{third} 
For every positive sequence $(\varepsilon_n)_n$ with limit $0$, there exists a \emph{surjective and four-valent} symbol $\varphi$ such that:
\begin{displaymath} 
a_{n} (C_\varphi) \lesssim \e^{- n\varepsilon_n} \, .
\end{displaymath} 
Consequently, there exists a \emph{surjective and four-valent} symbol $\varphi \colon \D\to \D$ such that the composition operator 
$C_\varphi \colon H^2 \to H^2$ is in every Schatten class $S_{p} (H^2)$, $p > 0$.
\end{theoreme}
\begin{proof} 
Observe first that $\Vert \varphi \Vert_\infty = 1$ when $\varphi$ is surjective, so that, in view of Theorem~\ref{basic}, we cannot dispense with the numbers 
$\varepsilon_n$, even if they can tend to $0$ arbitrarily slowly. 

Now, we can choose $\delta \colon (0, 1) \to (0, 1)$ non-decreasing such that $\delta (\varepsilon_n) \leq \e^{- n \varepsilon_n}$ for all $n$, and then, using 
Theorem~\ref{bet}, we get a  surjective and four-valent symbol $\varphi$, satisfying for all $h$ small enough:
\begin{displaymath} 
\rho_{\varphi} (h) \leq h \, \delta^{\,2} (h) \, .
\end{displaymath} 
Proposition~\ref{second} gives: 
\begin{displaymath} 
a_{n} (C_\varphi) \lesssim \inf_{0 < h < 1} \big[ \e^{- n h} + \delta (h) \big] \, .
\end{displaymath} 
Adjusting $h = \varepsilon_n$, we get $a_{n} (C_\varphi) \lesssim \e^{- n \varepsilon_n}$.  

To get the second part of the theorem, just take $\varepsilon_n = n^{- 1/2}$.
\end{proof}
%

%%%%%%%%%%%%%%%%%%%%%%%%%%%%%%%%%%%%
\subsection{A simplified proof of Theorem~\ref{first}}

We give here the announced simplified proof of Theorem~\ref{first}. This proof is based on the following key lemma, in which $\mathcal{H} (\D)$ denotes 
the set of holomorphic functions on $\D$.
\begin{lemme} \label{key} 
There exists a numerical constant $C$ such that, if $f \in \mathcal{H}(\D)$ satisfies, for some $\alpha \in \R$:
\begin{displaymath}
\left\{
\begin{array} {l}
\Im [f (0)] < \alpha \smallskip \\ 
f (\D) \subseteq \{z \in \C \, ; \ 0 < \Re z < \pi\} \cup \{z \in \C \, ; \ \Im z < \alpha \} \, ,
\end{array} 
\right.
\end{displaymath}
then:
\begin{displaymath} 
m (\{\Im f^{\ast} > y\}) \leq C \, \e^{\alpha - y} \, , \quad \text{for } y \geq \alpha \, .
\end{displaymath} 
\end{lemme}

We first show how this lemma allows us to conclude. 
\begin{proof} [Proof of Theorem~\ref{first}]
Let $g \colon (0, \infty) \to (0, \infty)$ be a continuous decreasing function such that:
\begin{displaymath} 
\lim_{t \to 0^{+}} g (t) = +\infty \, , \quad g (\pi) = \pi \, ,\quad \lim_{t \to +\infty} g (t) = 0 \, .
\end{displaymath} 
Then let $\Omega$ be the simply connected region defined by:
\begin{displaymath} 
\Omega = \{x + i y \, ;\ x > 0 \, , \quad g (x) < y < g (x) + 4 \pi\} \, ,
\end{displaymath} 
and $f \colon\D \to \Omega$ be a Riemann map such that $f (0) = \pi + 3 i \pi$. Observe that we can apply Lemma~\ref{key} to $f$ with $\alpha = 5\pi$ 
since $\Im f (0) = 3\pi$ and if $f (z) = x + i y$ with $x \geq \pi$; hence:
\begin{displaymath} 
\Im f (z) = y < g (x) + 4 \pi \leq g (\pi) + 4 \pi = 5 \pi \, . 
\end{displaymath} 
Finally, consider the symbol $\varphi = \e^{- f}$. It is nearly surjective: $\phi (\D) = \D \setminus \{0\}$, and two-valent, as easily checked.  

For $0 < h \leq 1/2$, we have for $\xi \in \T$ and $|\phi^\ast (\xi) | > 1 - h$:
\begin{displaymath} 
\e^{ - 2 h} \leq 1 - h < |\phi^\ast (\xi) | = \exp \big( - \Re  f^\ast (\xi) \big) \, ;
\end{displaymath} 
hence $\Re f^{\ast} (\xi) < 2 h$. 

But if $2 h > x = \Re f^{\ast} (\xi)$, we have $g (x) > g (2 h)$. As $f^{\ast} (\xi) = x + i y \in \overline{\Omega}$, we get   
$\Im f^{\ast} (\xi) = y \geq g (x) > g (2h)$.  Lemma~\ref{key} now gives:
\begin{equation} \label{kol} 
m (\{\xi \, ; \ |\varphi^{\ast} (\xi)| > 1 - h \}) \leq m (\{\xi \, ;\ \Im f^{\ast} (\xi) > g (2 h) \}) \leq C \, \e^{5 \pi  - g (2 h)} \, .
\end{equation}
It is now enough to adjust $g$ so as to have $\e^{g(t)}\geq C \, \e^{5 \pi} / \delta (t/2)$ for $t$ small enough to get \eqref{firs} from \eqref{kol}.
\end{proof}
\begin{proof} [Proof of Lemma~\ref{key}] 
We now prove  Lemma~\ref{key}. If $\e^{y - \alpha} < 2$, there is nothing to prove, since then:
\begin{displaymath} 
m (\Im f^{\ast} > y) \leq 1 \leq 2 \, \e^{\alpha - y} \, .
\end{displaymath} 
We can hence assume that $\e^{y - \alpha}\geq 2$.
First, we make a comment. If the Riemann mapping theorem is very general and flexible, it gives very few informations on the parametrization 
$t \mapsto f^{\ast} (\e^{it})$ when $f \colon \D\to \Omega$ is a conformal map, except in some specific cases (lens maps, cusps, etc.: see \cite{LIQURO}). 
Here, the Kolmogorov weak type inequality provides a substitute. Write:
\begin{displaymath} 
f = u + i v
\end{displaymath} 
and set:
\begin{displaymath} 
f_1 = - i f + i \frac{\pi}{2} - \alpha = v - \alpha + i \bigg( \frac{\pi}{2} - u \bigg) 
\end{displaymath} 
and:
\begin{displaymath} 
F_1 = 1 + \e^{f_1} = (1 + \e^{v - \alpha} \sin u) + i \e^{v - \alpha} \cos u \, .
\end{displaymath} 

If $v < \alpha$, then $\Re F_1 > 1 - |\sin u| \geq 0$. If $v \geq \alpha$, then $0 < u < \pi$ and $\Re F_1 \geq 1$. Hence $F_1$ maps $\D$ to the right 
half-plane $\C_0 = \{z \, ; \ \Re z > 0 \}$. Finally, let $F = U + i V \colon \D \to \C_0$ be defined by:
\begin{displaymath}
F = F_1 - i \Im F_{1}(0) \, , 
\end{displaymath}
so that $V (0) = 0$. By the Kolmogorov inequality for the conjugation map $U \mapsto V$, and the harmonicity of $U$, we have, for all $\lambda>0$ 
($a$ designating an absolute constant):
\begin{equation} \label{ko} 
m (|F^{\ast}| > \lambda) \leq \frac{a}{\lambda} \, \Vert U^{\ast} \Vert_1 
= \frac{a}{\lambda} \int_{\T} U^{\ast} \, dm = \frac{a}{\lambda} \, U(0) \, .
\end{equation}

Next, we claim that:
\begin{equation} \label{claim} 
|\Im F_{1} (0)|<1 \quad \text{and} \quad  U (0) < 2 \, .
\end{equation}
Indeed, $v (0) < \alpha$ by hypothesis, so that $ |\Im F_{1} (0)| = \e^{v (0) - \alpha} |\cos u (0)| < 1$, and $U (0) = 1 + \e^{v (0) - \alpha} \sin u (0) < 2$. 
Suppose now that, for some $y > \alpha$ and $z \in \D$, we have $v (z) > y$. Then, $0 <u (z) < \pi$ by our second assumption, and this implies 
$\Re \e^{f_{1} (z)} = \e^{v (z) - \alpha} \sin u (z) > 0$, so that, using $|1 + w| \geq |w|$ if $\Re w > 0$ and \eqref{claim}, and remembering that 
$\e^{y - \alpha} \geq 2$:
\begin{align*}
|F (z) | 
& = \big|1 + \e^{f_{1} (z)} - i \, \Im F_{1} (0) \big| \geq  \big|1 + \e^{f_{1} (z)} \big| - 1 \\ 
& \geq \big| \e^{f_{1} (z)} \big| - 1 = \e^{v (z) - \alpha} - 1 > \e^{y - \alpha} - 1 \geq \frac{1}{2} \, \e^{y - \alpha}\, .
\end{align*}
Taking radial limits and using \eqref{ko} and \eqref{claim}, we get:
\begin{displaymath}
m (\Im f^{\ast} > y ) \leq m (| F^{\ast}| >\e^{y - \alpha} / 2) \leq 4 a \, \e^{\alpha - y} \, .
\end{displaymath}
This  ends the proof of Lemma~\ref{key} with $C = \max (2, 4 a)$.
\end{proof}
%

%%%%%%%%%%%%%%%%%%%%%%%%%%%%%%%%%%%%%%%%%%%%%%%%%%%%%%%%%%%%%%
\section{Application to the multidimensional case}

In this section, we apply Theorem~\ref{bet} and Theorem~\ref{third} to show that, for $N \geq 2$, the image of the symbol cannot determine the behavior of 
the approximation numbers, or rather of $\beta_N (C_\phi)$, of the associated composition operator $C_\phi \colon H^2 (\D^N) \to H^2 (\D^N)$.
\smallskip

Recall that for an operator $T \colon H_1 \to H_2$, we set:
\begin{equation} 
\beta^-_N (T) = \liminf_{n \to \infty} [a_{n} (T)]^{1 / n^{1 / N}} \quad \text{and} \quad 
\beta^+_N (T) = \limsup_{n \to \infty} [a_{n} (T)]^{1 / n^{1 / N}}\, ,
\end{equation} 
and write $\beta_N (T)$ when $\beta^-_N (T) = \beta^+_N (T)$.
\begin{theoreme}
For $N \geq 2$, there exist pairs of symbols $\Phi_1, \Phi_2 \colon \D^N \to \D^N$, such that $\Phi_1 (\D^N) = \Phi_2 (\D^N)$ and:
\begin{itemize}
\setlength\itemsep{2 pt}
\item [$1)$] $C_{\Phi_1}$ is not bounded, but $C_{\Phi_2}$ is compact, and even $\beta_N (C_{\Phi_2}) = 0$; 

\item [$2)$] $C_{\Phi_1}$ is bounded but not compact, so $\beta_N (C_{\Phi_1}) = 1$, and $C_{\Phi_2}$ is compact, with $\beta_N (C_{\Phi_2}) = 0$; 

\item [$3)$] $C_{\Phi_1}$ is compact, with $\beta^-_N (C_{\Phi_1}) > 0$ and $\beta^+_N (C_{\Phi_1}) < 1$, and $C_{\Phi_2}$ is compact, with 
$\beta_N (C_{\Phi_2}) = 0$;

\item [$4)$] $C_{\Phi_1}$ is compact, with $\beta_N (C_{\Phi_1}) = 1$, and $C_{\Phi_2}$ is compact, but with $\beta_N (C_{\Phi_2}) = 0$. 
\end{itemize}
\end{theoreme}
\begin{proof}
Let $\sigma \colon \D \to \D$ be a surjective symbol such that $\rho_\sigma (h) \leq h^N \, \e^{- 2 / h^2}$ given by Theorem~\ref{bet}. 
By Proposition~\ref{seconde}, we have, with $\gamma = N - 2$:
\begin{displaymath} 
a_{n} (C_\sigma \colon \mathcal{B}_{\gamma} \to H^2) \lesssim \inf_{0 < h < 1} ({n^{(N - 1)/2}} \e^{- n h} + \e^{- 1 / h^2} ) \, \raise 1 pt \hbox{,} 
\end{displaymath} 
and, with $h = 1 / n^{1/3}$, we get $a_{n} (C_\sigma \colon \mathcal{B}_{\gamma} \to H^2) \lesssim \e^{- d \, n^{2 / 3}}$.  

We choose the exponent $2/3$ for fixing the ideas, but every exponent $\alpha > 1/2$, with $\alpha < 1$,  
(i.e. $a_{n} (C_\sigma \colon \mathcal{B}_{\gamma} \to H^2) \lesssim \e^{- d \, n^\alpha}$) would be suitable.
\smallskip

$1)$ We take $\Phi_1 (z_1, z_2, z_3, \ldots, z_N) = (z_1, z_1, \ldots, z_1)$. The composition operator $C_{\Phi_1}$ is not bounded because if 
$f_n (z_1, \ldots, z_N) = \big( \frac{z_1 + z_2}{2} \big)^n$, then 
$\| f_n \|_2^2 = 4^{- n} \sum_{k = 0}^n \binom{n}{k}^2 = 4^{- n} \binom{2 n}{n} \approx 1 / \sqrt{n}$, though 
$(C_{\Phi_1} f_n) (z_1, \ldots, z_N) = z_1^n$ and $\| C_{\Phi_1} f_n \|_2 = 1$. 

We define $\Phi_2$ by:
\begin{displaymath}
\Phi_2 (z_1, z_2, \ldots, z_N) = \big( \sigma (z_1), \sigma (z_1), \ldots, \sigma (z_1) \big) \, .
\end{displaymath}
Since $\sigma$ is surjective, we have $\Phi_2 (\D^N) = \Phi_1 (\D^N)$. 
Now, as in the proof of Theorem~\ref{cusp}, we have $C_{\Phi_2} = J C_\sigma M$, so:
\begin{displaymath}
a_n (C_{\Phi_2}) \leq a_{n} (C_\sigma \colon \mathcal{B}_{N - 2} \to H^2) \lesssim \e^{- d \, n^{2 / 3}} \, , 
\end{displaymath}
by the ideal property. Hence $[ a_n (C_{\Phi_2}) ]^{1 / n^{1 / N}} \lesssim \e^{- d \, n^{\frac{2}{3} - \frac{1}{N}}}$ and therefore 
$\beta_N (C_{\Phi_2}) = 0$ since $\frac{2}{3} - \frac{1}{N} > 0$.
\medskip

$2)$ We consider the lens map $\lambda = \lambda_{1/N}$ of parameter $1 / N$. 
We define:
\begin{displaymath}
\left\{ 
\begin{array}{l}
\Phi_1 (z_1, \ldots, z_N) = \big( \lambda (z_1), \lambda (z_1), \ldots, \lambda (z_1) \big) \smallskip \\
\Phi_2 (z_1, \ldots, z_N) = \big( \lambda [\sigma (z_1)], \lambda [\sigma (z_1)], \ldots, \lambda [\sigma (z_1)] \big) \, .
\end{array}
\right.
\end{displaymath}
Since $\sigma$ is surjective, we have $\Phi_1 (\D^N) = \Phi_2 (\D^N)$ and we saw in Theorem~\ref{lens} that $C_{\Phi_1}$ is bounded but not compact. 

On the other hand, we have the factorization $C_{\Phi_2} = J C_\sigma C_\lambda M$. Hence $C_{\Phi_2}$ is compact, and, as in $1)$, 
$\beta_N (C_{\Phi_2}) = 0$. 
\smallskip

$3)$ For this item, the map $\sigma$ does not suffice, and we will use another surjective symbol $s \colon \D \to \D$. 
By Theorem~\ref{bet}, there exists such a map $s$ with:
\begin{equation} \label{s1}
\rho_s (t) \leq t^2 \e^{- 2/t^2} 
\end{equation} 
and
\begin{equation} \label{s2}
\rho_s (t) \leq t \, \delta^{\,2} (t) 
\end{equation} 
for $t$ small enough, where $\delta \colon (0, 1) \to (0, 1)$ is a non-decreasing function such that $\delta (\eps_n) \leq \e^{- n \eps_n}$ and:
\begin{equation} \label{eps_n}
\eps_n = n^{- \frac{1}{4 N - 7}} \, . 
\end{equation} 
By the proof of Theorem~\ref{third}, \eqref{s2} implies that:
\begin{equation} \label{a_n (s)}
a_n (C_s) \leq \e^{- n \eps_n} \, .  
\end{equation} 
%
%\smallskip

We also consider a lens map $\lambda = \lambda_\theta$, with parameter $\theta < 1 / N$, and we set:
\begin{displaymath}
\left\{ 
\begin{array}{l}
\displaystyle \Phi_1 (z_1, \ldots, z_N) = \Big( \lambda (z_1), \lambda (z_1), \raise -1,5 pt \hbox{$\displaystyle \frac{z_3}{2}$} , \ldots , 
\raise -1,5 pt \hbox{$\displaystyle \frac{z_N}{2}$ }\Big) \smallskip \\
\displaystyle \Phi_2 (z_1, \ldots, z_N) = \Big( \lambda [s (z_1)], \lambda [s (z_1)], \raise -1,5 pt \hbox{$\displaystyle \frac{s (z_3)}{2}$} , \ldots , 
\raise -1,5 pt \hbox{$\displaystyle \frac{s (z_N)}{2}$} \Big) \, .
\end{array}
\right.
\end{displaymath}
Since $s$ is surjective, we have $\Phi_1 (\D^N) = \Phi_2 (\D^N)$. 
\smallskip

a) Let us prove that $\beta_N^- (C_{\Phi_1}) > 0$ and $\beta_N^+ (C_{\Phi_1}) < 1$. 

Note that:
\begin{displaymath} 
C_{\Phi_1} = C_u \otimes C_{v_3}\otimes \cdots \otimes C_{v_N} \, ,
\end{displaymath} 
where $u \colon \D^2 \to \D^2$ is defined by $u (z_1, z_2) = \big( \lambda (z_1), \lambda (z_1) \big)$ and $v_j \colon \mathbb{D} \to \mathbb{\D}$ is 
defined by $v_{j} (z_j) = z_j /2$. In fact, if $f \in H^2 (\D^2)$ and $g_j \in H^2 (\D)$, $3 \leq j \leq N$, we have:
%\goodbreak
%
\begin{align*}
[C_{\Phi_1} & (f \otimes g_3  \otimes \cdots \otimes g_N )] (z_1, z_2, z_3, \ldots, z_N)  \\
& =  (f \otimes g_3  \otimes \cdots \otimes g_N ) \big( u (z_1, z_2), v_3 (z_3), \ldots, v_N (z_N) \big) \\ 
& = f [ \lambda (z_1), \lambda (z_1) ] \, g_3 [v_3 (z_3)] \cdots g_N [v_N (z_N)] \\
& = (C_u f) (z_1, z_2) \, (C_{v_3} g_3) (z_3)  \cdots (C_{v_N} g_N) (z_N) \\
& = [(C_u \otimes C_{v_3} \otimes \cdots \otimes C_{v_N}) (f \otimes g_3 \otimes \cdots \otimes g_N)] (z_1, z_2, z_3, \ldots, z_N) \, ,
\end{align*}
hence the result since $H^2 (\D^2) \otimes H^2 (\D) \otimes \cdots \otimes H^2 (\D)$ is dense in $H^2 (\D^N)$. That proves in particular that $C_{\Phi_1}$ 
is compact since $C_u$ and $C_{v_3}, \ldots, C_{v_N}$ are (by Theorem~\ref{lens} for $C_u$).

By the supermultiplicativity of singular numbers of tensor products (see \cite[Lemma~3.2]{DHL}), it ensues that:
\begin{displaymath} 
a_{n^N} (C_{\Phi_1}) \geq a_{n^2} (C_u) \prod_{j = 3}^N a_{n} (C_{v_j}) = a_{n^2} (C_u) \, \Big( \frac{1}{2} \Big)^{n (N - 2)} \, .
\end{displaymath} 
By \cite[Remark at the end of Section~4]{DHL}, we have $a_{n^2} (C_u) \gtrsim \e^{ - b n}$ for some positive constant $b = b (\theta)$. Indeed, if 
$J = J_2 \colon H^2 (\D) \to H^2 (\D^2)$ is the canonical injection defined by $(J h) (z_1, z_2) = h (z_1)$ and $Q \colon H^2 (\D^2) \to H^2 (\D)$ is defined 
by $(Qf) (z_1) = f (z_1, 0)$, we have $C_\lambda = Q C_u J$. Hence $a_k (C_u) \gtrsim a_k (C_\lambda) \gtrsim \e^{- b \sqrt{k}}$. 

Therefore we get:
\begin{displaymath} 
a_{n^N} (C_{\Phi_1}) \gtrsim \e^{ - c n} 
\end{displaymath} 
for some positive constant depending only on $\theta$ and $N$. It follows that $\beta^-_N (C_{\Phi_1}) > 0$.
\medskip

To see that $\beta^+_N (C_{\Phi_1}) < 1$, we need the following lemma, whose proof is postponed.
\goodbreak

\begin{lemme} \label{uno} 
Let $S \colon H_1 \to H_1$ and $T \colon H_2 \to H_2$ be two operators between Hilbert spaces and $A, B$ a pair of positive numbers. Then, 
whenever:
\begin{displaymath} 
 a_{[n^A]}(S) \leq \e^{- c n} \quad \text{and} \quad  a_{[n^B]} (T) \leq \e^{- c n} \, , 
\end{displaymath} 
where $[\, . \,]$ stands for the integer part, we have, for some constant integer $M = M (A, B) > 0$:
\begin{displaymath} 
a_{M \, [n^{A + B}]} (S \otimes T) \leq \e^{- c n} \, .
\end{displaymath} 
\end{lemme} 
\smallskip

Let $S = C_u$ and $T = C_{v_3} \otimes \cdots \otimes C_{v_N}$. For $c$ small enough, we have $a_{n^{N - 2}} (T) \leq C \, (1/2)^n \leq \e^{- c n}$ and, 
using \eqref{mino lens}, $a_{n^2} (S) \leq \e^{- d n} \leq \e^{- c n}$. Hence, with $A = 2$, $B = N - 2$, Lemma~\ref{uno} gives:
\begin{displaymath} 
a_{M n^N} (C_{\Phi_1}) \lesssim \e^{- c n} \, .
\end{displaymath} 
Therefore $\beta^+_N (C_{\Phi_1}) \leq \e^{- c / M^{1/N}} < 1$. 
\smallskip

b) Define $\Psi \colon \D^N \to \D^N$ by:
\begin{displaymath} 
\Psi (z_1, z_2, z_3, \ldots, z_N) = \big( s (z_1), s (z_1), s (z_3), \ldots, s (z_N) \big) \,.
\end{displaymath} 

If $\tau_1 \colon \D^2 \to \D^2$ is defined by $\tau_{1} (z_1, z_2) = \big( s (z_1), s (z_1) \big)$ and the map 
$\tau_2 \colon \D^{N - 2} \to \D^{N - 2}$ by $\tau_{2} (z_3, \ldots, z_N) = \big( s (z_3), \ldots, s (z_N) \big)$, we have:
\begin{displaymath} 
C_\Psi = C_{\tau_1} \otimes C_{\tau_2} \, .
\end{displaymath} 
As in the proof of Theorem~\ref{cusp}, we have the factorization:
\begin{displaymath} 
\tau_1 \colon H^{2} (\D^2) \mathop{\longrightarrow}^M \mathcal{B}_{0} = {\mathcal B}^2 
\mathop{\longrightarrow}^{C_s} H^{2} (\D) \mathop{\longrightarrow}^J H^{2} (\D^2) \, .
\end{displaymath} 
Hence $a_n (C_{\tau_1}) \leq \| M \|\, \| J \|\, a_n (C_s \colon \mathcal{B}^2 \to H^2)$. 

By Proposition~\ref{seconde}, we have:
\begin{displaymath} 
a_n (C_s \colon \mathcal{B}^2 \to H^2) \lesssim \inf_{0 < h < 1} \bigg( \sqrt{n} \, \e^{- n h} + \sup_{0 < t \leq h} \sqrt{\frac{\rho_s (t)}{t^2}} \, \bigg) 
\, ;
\end{displaymath} 
so \eqref{s1} implies that $a_n (C_s \colon \mathcal{B}^2 \to H^2) \lesssim \inf_{0 < h < 1} ( \sqrt{n} \, \e^{- n h} + \e^{- 1 / h^2})$ and, taking 
$h = n^{- 1/3}$, we get, with some $c$ small enough:
\begin{displaymath} 
a_n (C_s \colon \mathcal{B}^2 \to H^2) \lesssim \e^{- c n^{2 / 3}} \, .
\end{displaymath} 
It follows that $a_n (C_{\tau_1}) \lesssim \e^{- c \, n^{2 / 3}}$ and hence:
\begin{equation} \label{tau1}
a_{[n^{3 / 2}]} (C_{\tau_1}) \lesssim \e^{- c \, n} \, .
\end{equation} 

On the other hand, \cite[Theorem~5.5]{BLQR} says that:
\begin{displaymath} 
a_n (C_{\tau_2}) \leq 2^{N - 3} \| C_s\|^{N - 2} \inf_{n_3 \cdots n_N \leq n} \big( a_{n_3} (C_s) + \cdots + a_{n_N} (C_s) \big) \, .
\end{displaymath} 
Taking $n_3 = \cdots = n_N = n^{\frac{1}{N - 2}}$, we get, using \eqref{a_n (s)}:
\begin{displaymath} 
a_n (C_{\tau_2}) \leq K^N N \, \exp \Big( - n^{\frac{1}{N  - 2}} \, \eps_{n^{\frac{1}{N  - 2}}} \Big)\, . 
\end{displaymath} 
Using \eqref{eps_n}, that gives:
\begin{displaymath} 
a_n (C_{\tau_2}) \, \lesssim \exp \big( - n^{\frac{1}{N - 2} ( 1 - \frac{1}{4N  - 7} ) } \big) = 
\exp \big( - n^{\frac{4}{4 N - 7}} \big) \, ,
\end{displaymath} 
or:
\begin{equation} \label{tau2}
a_{\big[n^{N - \frac{7}{4}} \big]} (C_{\tau_2}) \lesssim \e^{ - n} \leq \e^{- c n} \, . 
\end{equation} 

Now, \eqref{tau1} and \eqref{tau2} allow to use Lemma~\ref{uno} with $A = 3/2$ and $B = N - 7/4$, and we get:
\begin{displaymath} 
a_{M \, \big[n^{N - \frac{1}{4}} \big] } (C_\Psi) \lesssim \e^{- c n} \, . 
\end{displaymath} 
Equivalently:
\begin{displaymath} 
a_k (C_\Psi) \lesssim \exp \big( - c' k^{\frac{4}{4 N - 1}} \big) 
\end{displaymath} 
and:
\begin{displaymath} 
\big( a_k (C_\Psi) \big)^{1/ k^{1 / N}} \lesssim \exp \big( - c' k^{\frac{4}{4 N - 1} - \frac{1}{N}} \big) 
= \exp \big( - c' k^{\frac{1}{N (4 N  - 1)}} \big) \, ,
\end{displaymath} 
which gives $\beta_N (C_\Psi) = 0$.

To end the proof, it suffices to remark that $C_{\Phi_2} = C_\Psi \circ C_{\Phi_1}$, since $\Phi_2 = \Phi_1 \circ \Psi$, and hence 
$\beta_N^+ (C_{\Phi_2}) \leq \beta_N^+ (C_\Psi) = 0$, so $\beta_N (C_{\Phi_2}) = 0$.
\smallskip

$4)$ We use a Shapiro-Taylor map. This one-parameter map $\varsigma_\theta\,$, $\theta > 0$, was introduced by J. Shapiro and P. Taylor in 1973 (\cite{SHTA}) 
and was further studied, with a slightly different definition, in \cite[Section~5]{JFA}. J. Shapiro and P. Taylor proved that 
$C_{\varsigma_\theta} \colon H^2 \to H^2$ is always compact, but is Hilbert-Schmidt if and only if $\theta > 2$. Let us recall their definition. 

For $0 < \eps < 1$, we set $V_\eps = \{ z\in \C \, ; \ \Re z > 0 \text{ and } |z | < \eps \}$. For $\eps = \eps_\theta > 0$ small enough, one can define: 
\begin{displaymath} 
f_\theta (z) = z (- \log z )^\theta ,
\end{displaymath} 
for $z \in V_\eps$, where $\log z$ will be the principal determination of the logarithm. Let now $g_\theta$ be the conformal mapping from $\D$ onto $V_\eps$, 
which maps $\T = \partial \D$ onto $\partial V_\eps$, defined by $g_\theta (z) = \eps\, \phi_0 (z)$, where $\phi_0$ is given by:
\begin{displaymath} 
\phi_0 (z) = \frac{\displaystyle \Big( \frac{z - i}{i z - 1} \Big)^{1/2} - i} {\displaystyle - i \, \Big( \frac{z - i}{i z - 1} \Big)^{1/2} + 1} \, \cdot
\end{displaymath} 
Then, we define:
\begin{displaymath} 
\varsigma_\theta = \exp ( - f_\theta \circ g_\theta) .
\end{displaymath} 

We proved in \cite[Section~4.2]{LIQURO} (though it is not sharp) that:
\begin{equation} 
a_n (C_{\varsigma_\theta}) \gtrsim \frac{1}{n^{\theta/2}} \, \cdot
\end{equation} 

We define $\Phi_1 \colon \D^N \to \D^N$ as:
\begin{equation} 
\Phi_1 (z_1, z_2, \ldots, z_N) = \big( \varsigma_\theta (z_1), 0, \ldots, 0\big) \, .
\end{equation} 

If $J = J_N \colon H^2 (\D) \to H^2 (\D^N)$ is the canonical injection defined by $(Jh) (z_1, \ldots, z_N) = h (z_1)$ and 
$Q = Q_N \colon H^2 (\D^N) \to H^2 (\D)$ is defined by $(Q f) (z_1) = f (z_1, 0, \ldots, 0)$, then $C_{\Phi_1} = J C_{\varsigma_\theta} Q$; hence 
$C_{\Phi_1}$ is compact. On the other hand, we also have $Q C_{\Phi_1} J = C_{\varsigma_\theta}$, which implies that 
$a_n (C_{\Phi_1}) \gtrsim a_n (C_{\varsigma_\theta}) \gtrsim n^{- \theta/2}$. It follows that:
\begin{displaymath} 
\beta_N (C_{\Phi_1}) \geq \lim_{n \to \infty} (n^{- \theta/2})^{1/ n^{1/ N}} = 1 \, ,
\end{displaymath} 
and hence $\beta_N (C_{\Phi_1}) = 1$.

Now, if:
\begin{displaymath} 
\Phi_2 (z_1, \ldots, z_N) = \big( \varsigma_\theta [\sigma (z_1)], 0, \ldots, 0\big) \, , 
\end{displaymath} 
since $\sigma$ is surjective, we have $\Phi_1 (\D^N) = \Phi_2 (\D^N)$. Moreover, we have  
$C_{\Phi_2} = J C_{\varsigma_\theta \circ \sigma} Q = J C_\sigma C_{\varsigma_\theta} Q$, 
so $a_n (C_{\Phi_2}) \lesssim a_n (C_\sigma)$. Since $\rho_\sigma (h) \leq h^{N + 1} \, \e^{- 2 / h^2}$, Proposition~\ref{second} gives, with 
$h = 1/ n^{1/3}$:
\begin{displaymath} 
a_n (C_\sigma) \lesssim \e^{- c n^{2/3}} \, ,
\end{displaymath} 
so $[a_n (C_{\Phi_2})]^{1/n^{1/N}} \lesssim \exp ( - c \, n^{\frac{2}{3} - \frac{1}{N}})$ and $\beta_N (C_{\Phi_2}) = 0$.
\end{proof}
\begin{proof} [Proof of Lemma~\ref{uno}] 
In \cite{DHL}, we observed that the singular numbers of $S \otimes T$ are the non-increasing 
rearrangement of the numbers $s_j t_k$, where $s_j$ and $t_k$ denote respectively the $j$-th and the $k$-th singular number of $S$ and $T$.  We can assume 
$s_1 = t_1 = 1$. Using this observation, we will majorize the number of pairs $(j, k)$ such that $s_j t_k > \e^{- c n}$. Let $(j, k)$ be such a pair. Since 
$s_j \leq s_1 = 1$, we have $t_k \geq \e^{- c n}$ so that $k \leq [n^B] \leq n^B$. Hence, for some $2 \leq l \leq n$, we have 
$(l - 1)^B < k \leq l^B$. Then, due to the assumption on $T$, $t_k < \e^{- c (l - 1)}$ and $s_j \geq \e^{- c n} t_{k}^{- 1} \gtrsim \e^{- c (n - l + 1)}$, implying 
that $j \lesssim (n - l + 1)^{A}$, thanks to the assumption on $S$. As a consequence, since the number of integers $k$ such that $(l - 1)^B < k \leq l^B$ is 
dominated by $l^{B - 1}$, the number $\nu_n$ of pairs $(j, k)$ such that $s_j t_k > \e^{- c n}$ is dominated by:
\begin{displaymath} 
\sum_{l = 1}^n (n - l + 1)^{A} l^{B - 1} \sim n^{A+B} \int_{0}^1 t^{A} (1 - t)^{B} \, dt \, ,
\end{displaymath} 
by a Riemann sum argument. Next, let $M \in \N$ big enough to have:
\begin{displaymath} 
\sum_{l = 1}^n (n - l + 1)^A l^{B - 1} \leq M n^{A + B} - 1 \, , \quad \text{for all } n \, .
\end{displaymath} 
By definition, $a_{M [n^{A + B}]} (S \otimes T) \leq a_{\nu_{n} + 1} (S\otimes T) \leq \e^{- c n}$, giving the result.
\end{proof}
\medskip

\noindent {\bf Acknowledgement.} This paper was made when the two first-named authors visited the University of Sevilla in February 2018. It is their 
pleasure to thank this university and all colleagues therein for their warm welcome. 

The third-named author is partially supported by the project MTM2015-63699-P (Spanish MINECO and FEDER funds).

%%%%%%%%%%%%%%%%%%%%%%%%%%%%%%%%%%%%%%%%%%%%%%%%%%%%%%%%%%%%%%%%%%%%%%%%%%%%

 %%%%%%%%%%%%%%%%%%%%%%%%%%%%%%%%%%%%%%%%%%%%%%%%%%%%%%%%%%
\smallskip\goodbreak

{\footnotesize
Daniel Li \\ 
Univ. Artois, Laboratoire de Math\'ematiques de Lens (LML) EA~2462, \& F\'ed\'eration CNRS Nord-Pas-de-Calais FR~2956, 
Facult\'e Jean Perrin, Rue Jean Souvraz, S.P.\kern 1mm 18 
F-62\kern 1mm 300 LENS, FRANCE \\
daniel.li@euler.univ-artois.fr
\smallskip

Herv\'e Queff\'elec \\
Univ. Lille Nord de France, USTL,  
Laboratoire Paul Painlev\'e U.M.R. CNRS 8524 \& F\'ed\'eration CNRS Nord-Pas-de-Calais FR~2956 
F-59\kern 1mm 655 VILLENEUVE D'ASCQ Cedex, FRANCE \\
Herve.Queffelec@univ-lille.fr
\smallskip
 
Luis Rodr{\'\i}guez-Piazza \\
Universidad de Sevilla, Facultad de Matem\'aticas, Departamento de An\'alisis Matem\'atico \& IMUS,  
Calle Tarfia s/n \\ 
41\kern 1mm 012 SEVILLA, SPAIN \\
piazza@us.es
}

\end{document}